\documentclass[a4paper, 12pt]{amsart}

\usepackage[usenames,dvipsnames]{color}
\usepackage{amsthm,amsfonts,amssymb,amsmath,amsxtra}
\usepackage[all]{xy}
\SelectTips{cm}{}
\usepackage{xr-hyper}
\usepackage[colorlinks=
   citecolor=Black,
   linkcolor=Red,
   urlcolor=Blue]{hyperref}
\usepackage{verbatim}

\usepackage[margin=1.25in]{geometry}
\usepackage{mathrsfs}
\usepackage{bbm}

\RequirePackage{xspace}
\RequirePackage{etoolbox}
\RequirePackage{varwidth}
\RequirePackage{enumitem}
\RequirePackage{tensor}
\RequirePackage{mathtools}
\RequirePackage{longtable}
\RequirePackage{multirow}

\def\ge{\geqslant}
\def\le{\leqslant}
\def\a{\alpha}
\def\b{\beta}
\def\g{\gamma}

\def\s{\sigma}
\def\t{\tau}
\def\th{\theta}
\def\k{\kappa}
\def\l{\lambda}

\def\i{^{-1}}
\def\rig{\text{rig}}

\def\<{\langle}
\def\>{\rangle}

\newcommand{\sA}{\ensuremath{\mathscr{A}}\xspace}

\newcommand{\sR}{\ensuremath{\mathscr{R}}\xspace}

\newcommand{\fka}{\ensuremath{\mathfrak{a}}\xspace}

\newcommand{\fkR}{\ensuremath{\mathfrak{R}}\xspace}

\def\Tr{\mathrm{Tr}}

\newcommand{\BC}{\ensuremath{\mathbb {C}}\xspace}

\newcommand{\BF}{\ensuremath{\mathbb {F}}\xspace}
\newcommand{{\BG}}{\ensuremath{\mathbb {G}}\xspace}

\newcommand{{\BK}}{\ensuremath{\mathbb {K}}\xspace}

\newcommand{\BN}{\ensuremath{\mathbb {N}}\xspace}

\newcommand{\BR}{\ensuremath{\mathbb {R}}\xspace}
\newcommand{\BS}{\ensuremath{\mathbb {S}}\xspace}

\newcommand{\BZ}{\ensuremath{\mathbb {Z}}\xspace}

\newcommand{\CI}{\ensuremath{\mathcal {I}}\xspace}

\newcommand{\CK}{\ensuremath{\mathcal {K}}\xspace}

\newcommand{\CO}{\ensuremath{\mathcal {O}}\xspace}
\newcommand{\CP}{\ensuremath{\mathcal {P}}\xspace}
\newcommand{\CQ}{\ensuremath{\mathcal {Q}}\xspace}

\newcommand{\Ad}{{\mathrm{Ad}}}

\DeclareMathOperator{\Gal}{Gal}

\DeclareMathOperator{\Hom}{Hom}

\let\Im\relax
\DeclareMathOperator{\Im}{Im}

\DeclareMathOperator{\rank}{rank}

\def\tW{\tilde W}

\def\kk{\mathbf k}

\newtheorem{theorem}{Theorem}
\newtheorem{theoremA}{Theorem}

\newtheorem{proposition}[theorem]{Proposition}
\newtheorem{lemma}[theorem]{Lemma}

\newtheorem{corollary}[theorem]{Corollary}

\theoremstyle{definition}

\newtheorem{remark}[theorem]{Remark}

\numberwithin{equation}{section}
\numberwithin{theorem}{section}

\setitemize[0]{leftmargin=*,itemsep=\the\smallskipamount}
\setenumerate[0]{leftmargin=*,itemsep=\the\smallskipamount}

\renewcommand{\to}{%
   \ifbool{@display}{\longrightarrow}{\rightarrow}%
   }
\let\shortmapsto\mapsto
\renewcommand{\mapsto}{%
   \ifbool{@display}{\longmapsto}{\shortmapsto}%
   }
\newlength{\olen}
\newlength{\ulen}
\newlength{\xlen}
\newcommand{\xra}[2][]{%
   \ifbool{@display}%
      {\settowidth{\olen}{$\overset{#2}{\longrightarrow}$}%
       \settowidth{\ulen}{$\underset{#1}{\longrightarrow}$}%
       \settowidth{\xlen}{$\xrightarrow[#1]{#2}$}%
       \ifdimgreater{\olen}{\xlen}%
          {\underset{#1}{\overset{#2}{\longrightarrow}}}%
          {\ifdimgreater{\ulen}{\xlen}%
             {\underset{#1}{\overset{#2}{\longrightarrow}}}
             {\xrightarrow[#1]{#2}}}}%
      {\xrightarrow[#1]{#2}}
   }
\makeatother
\newcommand{\xyra}[2][]{%
   \settowidth{\xlen}{$\xrightarrow[#1]{#2}$}%
   \ifbool{@display}%
      {\settowidth{\olen}{$\overset{#2}{\longrightarrow}$}%
       \settowidth{\ulen}{$\underset{#1}{\longrightarrow}$}%
       \ifdimgreater{\olen}{\xlen}%
          {\mathrel{\xymatrix@M=.12ex@C=3.2ex{\ar[r]^-{#2}_-{#1} &}}}%
          {\ifdimgreater{\ulen}{\xlen}%
             {\mathrel{\xymatrix@M=.12ex@C=3.2ex{\ar[r]^-{#2}_-{#1} &}}}
             {\mathrel{\xymatrix@M=.12ex@C=\the\xlen{\ar[r]^-{#2}_-{#1} &}}}}}%
      {\mathrel{\xymatrix@M=.12ex@C=\the\xlen{\ar[r]^-{#2}_-{#1} &}}}%
   }
\makeatletter
\newcommand{\xla}[2][]{%
   \ifbool{@display}%
      {\settowidth{\olen}{$\overset{#2}{\longleftarrow}$}%
       \settowidth{\ulen}{$\underset{#1}{\longleftarrow}$}%
       \settowidth{\xlen}{$\xleftarrow[#1]{#2}$}%
       \ifdimgreater{\olen}{\xlen}%
          {\underset{#1}{\overset{#2}{\longleftarrow}}}%
          {\ifdimgreater{\ulen}{\xlen}%
             {\underset{#1}{\overset{#2}{\longleftarrow}}}
             {\xleftarrow[#1]{#2}}}}%
      {\xleftarrow[#1]{#2}}
   }
\newcommand{\isoarrow}{%
   \ifbool{@display}{\overset{\sim}{\longrightarrow}}{\xrightarrow\sim}%
   }
   
\begin{document}

\title[Newton decomposition]{Cocenters of $p$-adic groups, I: Newton decomposition}
\author[X. He]{Xuhua He}
\address{Department of Mathematics, University of Maryland, College Park, MD 20742 and Institute for Advanced Study, Princeton, NJ 08540}
\email{xuhuahe@math.umd.edu}

\thanks{X. H. was partially supported by NSF DMS-1463852.}


\date{\today}

\begin{abstract}
In this paper, we introduce the Newton decomposition on a connected reductive $p$-adic group $G$. Based on it we give a nice decomposition of the cocenter of its Hecke algebra. Here we consider both the ordinary cocenter associated to the usual conjugation action on $G$ and the twisted cocenter arising from the theory of twisted endoscopy. 

We give Iwahori-Matsumoto type generators on the Newton components of the cocenter. Based on it, we prove a generalization of Howe's conjecture on the restriction of (both ordinary and twisted) invariant distributions. Finally we give an explicit description of the structure of the rigid cocenter. 
\end{abstract}

\maketitle


\section*{Introduction}

\subsection{}\label{BDK} A basic philosophy in representation theory is that {\it characters tell all}. This is the case for finite groups. More precisely, let $G$ be a finite group. Let $H=\BC[G]$ be its group algebra and $\fkR(G)$ be the Grothendieck group of finite dimensional complex representations of $G$. Then the trace map induces an isomorphism $$\Tr: \bar H \to \fkR(G)^*,$$ where $\bar H=H/[H, H]$ is the cocenter of $H$. In other words, for finite groups, the cocenter is ``dual'' to representations. 

Now how about $p$-adic groups?

Let $\BG$ be a connected reductive group over a nonarchimedean local field $F$ of arbitrary characteristic and $G=\BG(F)$. Let $H$ be the Hecke algebra of compactly supported, locally constant, $\BC$-valued functions on $G$. Let $\fkR(G)$ be the Grothendieck group of smooth admissible complex representations of $G$. In \cite{BDK} and \cite{Kaz}, Bernstein, Deligne and Kazhdan established the duality between the cocenter and the representations of $G$ in the following sense: $$\Tr: \bar H \xrightarrow{\cong} \fkR(G)^*_{good},$$ where $\fkR(G)^*_{good}$ is the space of ``good forms'' on $\fkR(G)$. Such a relation is further studied by Dat in \cite{Dat}. 

One may also consider the twisted version arising from the theory of twisted endoscopy. The recent work of Henniart and Lemaire \cite{HL} established the duality between the cocenter and the representations in the twisted version.  

\subsection{} The main purpose of this paper is to investigate the structure of the cocenter of the Hecke algebra $H$. Here we consider both the ordinary and the twisted cocenter, and the Hecke algebra we consider is not an algebra of $\BC$-valued functions, but its integral form, i.e., the algebra of $\BZ[p \i]$-valued functions instead, where $p$ is the residual characteristic of $F$. The structure of the integral form helps us to understand not only the ordinary representations of $p$-adic groups, but the mod-$l$ representations as well. It will be used to study the relation between the cocenter and the mod-$l$ representations in a future joint work with Ciubotaru \cite{CH2}.

Note that invariant distributions on $G$ are linear function on the cocenter of $H$. Thus knowledge of the structure of $\bar H$ will also help us to understand the invariant distributions. We will discuss some application in this direction later in this paper. 

\subsection{} For the group algebra of a finite group, the structure of the cocenter is very simple: it has a standard basis indexed by the set of conjugacy classes. 

The main difference between the Hecke algebra of a connected reductive $p$-adic group and the group algebra of a finite group comes from the two conditions in the definition of Hecke algebra:

\begin{itemize}
\item The ``locally constant'' condition, which means that in the cocenter, one can not separate a single conjugacy class from the others. This suggests that we should seek for a decomposition of the $p$-adic group $G$ into open subsets, each of which is a union of (ordinary or twisted) conjugacy classes. 

\item The ``compact support'' condition, which suggests that the sought-after open subset should be of the form $G \cdot X$, where $\cdot$ is the conjugation action and $X$ is an open compact subset of $G$. 
\end{itemize}

\subsection{} In this paper, we introduce the decomposition of $G$ into Newton strata, which satisfy the desired properties mentioned above. 

\begin{theoremA}[see Theorem \ref{newton-gf}]\label{A}
We have the Newton decomposition $$G=\sqcup_\nu G(\nu).$$ Here each Newton stratum $G(\nu)$ is of the form $G \cdot X_
\nu$, where $X_\nu$ is an open compact subset of $G$. 
\end{theoremA}

Let us provide some background on the Newton strata. We may realize $G$ as $\BG(\breve F)^\s$, where $\breve F$ is the completion of the maximal unramified extension of $F$ and $\s$ is the morphism on $\BG(\breve F)$ induced by the Frobenius morphism of $\breve F$ over $F$. The $\s$-twisted conjugacy classes of $\BG(\breve F)$ are classified by Kottwitz in \cite{Ko1} and \cite{Ko2}, in terms of the Newton points together with the image under the Kottwitz map. For split groups, by taking the intersection of $G$ with $\s$-twisted conjugacy classes of $\BG(\breve F)$, we obtain a decomposition of $G$. The situation is more subtle if the group is not quasi-split, as the Newton map of $G$ does not coincide with the Newton map of $\BG(\breve F)$. The difficulty is overcome by the Iwahori-Matsumoto type generators which we discuss later in \S\ref{0-IM}. 

\smallskip

For any Newton point $\nu$, let $H(\nu)$ be the subspace of $H$ consisting of functions supported in $G(\nu)$ and $\bar H(\nu)$ be its image in $\bar H$. We obtain the desired decompositions for $H$ and $\bar H$. 

\begin{theoremA} [see Theorem \ref{newton-h}]\label{B}
We have the Newton decomposition for the Hecke algebra $H$ and its cocenter $\bar H$: $$H=\oplus_\nu H(\nu), \qquad \bar H=\oplus_\nu \bar H(\nu).$$
\end{theoremA}

\subsection{} Let $\CK$ be an open compact subgroup of $G$ and $H(G, \CK)$ be the Hecke algebra of compactly supported, $\CK$-biinvariant functions on $G$. We have $H=\varinjlim\limits_{\CK} H(G, \CK)$. To understand the representations of $G$, we need to understand not only the structure of the cocenter of $H$, but the structure of the cocenter of $H(G, \CK)$ as well. Unfortunately, for any given $\CK$, $H(G, \CK)$ does not have the Newton decomposition as the Newton strata of $G$ are not stable under the left/right action of $\CK$. However, we show that 

\begin{theoremA}[see Theorem \ref{newton-hn}]\label{C}
Let $n \in \BN$ and $\CI_n$ be the $n$th congruence subgroup of a Iwahori subgroup $\CI$ (cf. \S \ref{4.2}). Then 

(1) The cocenter of $H(G, \CI_n)$ has the desired Newton decomposition $$\bar H(G, \CI_n)=\oplus_{\nu} \bar H(G, \CI_n; \nu).$$ 

(2) For each $\nu$, any element in the Newton component $\bar H(G, \CI_n; \nu)$ is represented by a function of $H(G, \CI_n)$ that is supported in the open compact subset $X_\nu$. 
\end{theoremA} 

This is the main result of this paper. The two key ingredients of the proof are

\begin{itemize}
\item Newton decompositions on $G$, $H$ and $\bar H$ that we discussed above.

\item Iwahori-Matsumoto type generators that we are going to discuss in \S\ref{0-IM}. 
\end{itemize}

\smallskip

In the body of the paper, we consider a more general case by allowing twists by an automorphism $\th$ of $G$ and a character $\omega$ of $G$. Theorem \ref{A}, \ref{B} and \ref{C} are proved under this general setting. 

\subsection{}\label{0-IM} Let $\CI$ be an Iwahori subgroup of $G$ and $\tW$ be the Iwahori-Weyl group. Then we have the decomposition $G=\sqcup_{w \in \tW} \CI \dot w \CI$. It is known that the Iwahori-Hecke algebra $H(G, \CI)$ has the Iwahori-Matsumoto presentation with basis given by $\{T_w\}_{w \in \tW}$, where $T_w$ is the characteristic function on $\CI \dot w \CI$. 

In the joint work with Nie \cite{HN1}, we discovered that the cocenter $\bar H(G, \CI)$ has a standard basis $\{T_w\}$, where $w$ runs over minimal length representatives of conjugacy classes of $\tW$. This is the Iwahori-Matsumoto type basis of the cocenter of Iwahori-Hecke algebra $H(G, \CI)$. 

Now come back to our general situation $H(G, \CK)$, where $\CK$ is a congruent subgroup of $\CI$. We define $X_\nu=\cup_w \CI \dot w \CI$, where $w \in \tW$ runs over the minimal length elements in the conjugacy class associated to the given Newton point $\nu$. We show that $X_\nu$ is the sought-after open compact subset in Theorem \ref{A} and Theorem \ref{C}. The proof is based on 
\begin{itemize}
\item Some remarkable properties on the minimal length elements established in \cite{HN1}.

\item The compatibility between the reduction method on $H(G, \CK)$ in \S\ref{4} and the reduction method on $G$ introduced in \cite{He-14}.
\end{itemize}

\subsection{} Now we discuss some applications. 

In \cite{Howe}, Howe conjectured that for any open compact subgroup $\CK$ and compact subset $X$ of $G$, the restriction of invariant distributions $J(G \cdot X)$ supported in $G \cdot X$ to $H(G, \CK)$ is finite dimensional. This conjectured is proved by Clozel \cite{Cl} over $p$-adic fields and later by Barbasch and Moy \cite{BM} over any nonarchimedean local field. The twisted invariant distribution has not been much studied yet. 

Howe's conjecture plays a fundamental role in the harmonic analysis of $p$-adic groups (see e.g. \cite{HC} and \cite{De}). It will also play a crucial role in the future joint work with Ciubotaru \cite{CH2} in the proof of trace Paley-Wiener theorem for mod-$l$ representations of $p$-adic groups. 

In Section \ref{5}, we give a different proof of the Howe's conjecture, which is valid for both the ordinary and the twisted invariant distributions. Note that $G \cdot X$ is contained in a finite union of $G(\nu)$ and the ordinary/twisted invariant distributions supported in $G(\nu)$ are linear functions on $\bar H(\nu)$. It is also easy to see that for any given Newton point $\nu$, there are only finitely many minimal length elements $w$ associated to it. Now Howe's conjecture follows from the Newton decomposition of $\bar H(G, \CK)$ and the Iwahori-Matsumoto type generators of $\bar H(G, \CK; \nu)$. 

\begin{theoremA} [see Theorem \ref{howe-1}]
For any open compact subgroup $\CK$ and compact subset $X$ of $G$, $$\dim J(G \cdot X) \mid_{H(G, \CK)}<\infty.$$
\end{theoremA}
 
In fact, in the proof we mainly use the part $\bar H(G, \CK)=\sum_{\nu} \bar H(G, \CK; \nu)$. The fact that this is a direct sum will play an important role later (see e.g. \cite{CH2}) when we study the relation between the cocenter and the representations. 
 
\subsection{} As another application, we describe the structure of the rigid cocenter in terms of generators and relations. 

As discussed in \S\ref{BDK}, the cocenter is ``dual'' to the representations. Based on the Newton decomposition, we may decompose the whole cocenter into the rigid part and non-rigid part. The rigid cocenters of various Levi subgroups form the ``building blocks'' of the whole cocenter. 

We have mentioned in \S \ref{0-IM} (see also Theorem \ref{newton-hn}) that the cocenter has the Iwahori-Matsumoto type generators. The next problem is to describe the relations between these generators. In Theorem \ref{relation}, we solve this problem for the rigid cocenter: the Iwahori-Matsumoto type generators of the rigid cocenter are given by the cocenters of Hecke algebras of parahoric subgroups $\CP$, and the relations between these generators are given by the $(\CP, \CQ, x)$-graphs (see \S\ref{pqx} for the precise definition). The proof is based on the Iwahori-Matsumoto type generators of the cocenter and Howe's conjecture. 

It is shown in \cite{CH} that the rigid cocenter of affine Hecke algebras plays an important role in the study of representations of affine Hecke algebras. We expect that the rigid cocenter of the Hecke algebra $H$ plays a similar role in the study of ordinary and mod-$l$ representations of $p$-adic groups. 
 
\subsection{Acknowledgment} The study of the cocenter of Hecke algebras of $p$-adic groups began after Dan Ciubotaru and I finished the paper \cite{CH} on the rigid cocenter of affine Hecke algebras. I would like to thank him for many enjoyable discussions on the cocenter project and for drawing my attention to Howe's conjecture. I would like to thank Thomas Haines, Ju-Lee Kim, George Lusztig and Sian Nie for useful discussions and comments. I also would like to thank the referees for the helpful comments, and for pointing out a mistake in an earlier version of the proof of Lemma \ref{qq}. 

\section{Preliminaries}

\subsection{} Let $F$ be a nonarchimedean local field of arbitrary characteristic. Let $\CO_F$ be its valuation ring and $\kk=\BF_q$ be its residue field. Let $\BG$ be a connected reductive group over $F$ and $G=\BG(F)$. Fix a maximal $F$-split torus $A$. Let $\mathscr A$ be the apartment corresponding to $A$. Fix an alcove $\fka_C\subset \sA$, and denote by $\CI$ the associated Iwahori subgroup. 

Let $Z$ be the centralizer of $A$ and $N_G A$ be the normalizer of $A$. Denote by $W_0=N_G A(F)/Z(F)$ the relative Weyl group. The {\it Iwahori-Weyl group} $\tW$ is defined to be $\tW=N_G A(F)/Z_0$, where $Z_0$ is the unique parahoric subgroup of $Z(F)$ (cf. \cite[\S 5.2.7]{BT2}). The group $\tW$ acts on $\mathscr A$ by affine transformations as described in \cite[\S 1]{Tits}. Let $G_0$ be the subgroup of $G$ generated by all parahoric subgroups, and define $N_0 A=G_0\cap N_G A(F)$. Let $\tilde \BS$ be the set of simple reflections at the walls of $\fka_C$. By Bruhat and Tits \cite[Prop. 5.2.12]{BT2}, the quadruple $(G_0,\CI,N_0 A, \tilde \BS)$ is a (double) Tits system with affine Weyl group $W_a=N_0 A/N_0 A\cap \CI$. We have a semidirect product
$$\tW=W_a \rtimes \Omega,$$ where $\Omega$ is the stabilizer of the alcove $\fka_C$ in $\tW$. Thus $\tW$ is a quasi-Coxeter system and is equipped with a Bruhat order $\leq$ and a length function $\ell$. 

For any $K \subset \tilde \BS$, let $W_K$ be the subgroup of $\tW$ generated by $s \in K$. Let ${}^K \tW$ be the set of elements $w \in \tW$ of minimal length in the cosets $W_K w$. 

Set $V=X_*(Z)_{\Gal(\bar F/F)} \otimes \BR$, where $\bar F$ is the completion of a separable closure of $F$. By choosing a special vertex of $\fka_C$, we may identify $\sA$ with the underlying affine space of $V$ and by \cite[Proposition 13]{HR}, $$\tW \cong X_*(Z)_{\Gal(\bar F/F)} \rtimes W_0=\{t^\l w; \l \in X_*(Z)_{\Gal(\bar F/F)}, w \in W_0\}.$$

\subsection{}\label{hecke}
Let $\CI'$ be the pro-$p$ Iwahori subgroup of $G$. Following \cite[Chapter I, \S 2]{Vi}, we define a $\BZ[\frac{1}{p}]$-valued Haar measure $\mu_G$ by $$\mu_G (\CK, \mu)=\frac{[\CK: \CK \cap \CI']}{[\CI': \CK \cap \CI']}, \quad \text{ for any open compact subgroup } \CK \text{ of } G.$$ Note that $[\CI': \CK \cap \CI']$ is a power of $p$. Thus $\mu_G (\CK, \mu) \in \BZ[\frac{1}{p}]$ for all $\CK$. 

Let $H=H(G)$ be the space of locally constant, compactly supported $\BZ[\frac{1}{p}]$-valued functions on $G$. For any open compact subgroup $\CK$ of $G$, let $H(G, \CK)$ be the space of compactly supported, $\CK \times \CK$-invariant $\BZ[\frac{1}{p}]$-valued functions on $G$. Then for $\CK' \subset \CK$, we have a natural embedding $H(G, \CK) \hookrightarrow H(G, \CK')$ and $$H=\varinjlim\limits_{\CK} H(G, \CK).$$ 
Note that for any open compact subgroup $\CK$, $H(G, \CK)$ has a canonical $\BZ[\frac{1}{p}]$-basis $\{\mathbbm{1}_{\CK K g \CK}; g \in \CK \backslash G/\CK\}$, where $\mathbbm{1}_{\CK g \CK}$ is the characteristic function on $\CK g \CK$. 

The space $H$ is equipped with a natural convolution product $$f f'(g)=\int_{G} f(x) f'(x \i g) d \mu, \quad \text{ for } f, f' \in H, g \in G.$$ It can be described in a more explicit way as follows. 

(a) Let $X, Y$ be open compact subsets of $G$ and $\CK$ be an open compact pro-$p$ subgroup of $G$ such that $\CK X=X$ and $Y \CK=Y$. Then $$\mathbbm{1}_X \mathbbm{1}_Y=\sum_{g \in \CK \backslash X Y/\CK} \frac{\mu_{G \times G}(p \i (\CK g \CK))}{\mu_G(\CK g \CK)} \mathbbm{1}_{\CK g \CK},$$ where $p: X \times Y \to X Y$ is the multiplication map. 

Since the volume of each double coset of $\CK$ is a power of $p$, $\frac{\mu_G(p \i (\CK g \CK))}{\mu_G(\CK g \CK)} \in \BZ[\frac{1}{p}]$. 

\subsection{} Let $R$ be a commutative $\BZ[\frac{1}{p}]$-algebra and $H_R=H \otimes_{\BZ[\frac{1}{p}]} R$. We define the twisted action of $G$ on $H_R$ as follows. Let $\th$ be an automorphism of $G$ that stabilizes $A$ and $\CI$. We denote the induced affine transformation on $V$ and the length-preserving automorphism on $\tW$ still by $\th$. We assume furthermore that the actions of $\th$ on $V$ and on $\tW$ are both of finite order. Let $\omega$ be a character of $G$, i.e., a homomorphism from $G$ to $R^\times$ whose kernel is an open subgroup of $G$. We define the $(\th, \omega)$-twisted $G$-action on $H_R$ by $${}^x f(g)=\omega(x) f(x \i g \th(x)) \qquad \text{ for } f \in H_R, x, g \in G.$$

We define the $(\th, \omega)$-twisted commutator of $H_R$ as follows. 

Notice that $H_R$ is generated by the elements $\mathbbm{1}_X$, where $X$ is an open compact subset of $G$ such that $\omega \mid_X$ is constant. Let $[H_R, H_R]_{\th, \omega}$ be the $R$-submodule of $H_R$ spanned by $$[\mathbbm{1}_X, \mathbbm{1}_{X'}]_{\th, \omega}=\mathbbm{1}_X \mathbbm{1}_{X'}-\omega(X) \i \mathbbm{1}_{X'} \mathbbm{1}_{\th(X)},$$ where $X, X'$ are open subsets of $G$ such that $\omega \mid_X$ is constant. 

\begin{proposition}\label{comm}
The $R$-submodule $[H_R, H_R]_{\th, \omega}$ of $H_R$ equals the $R$-submodule of $H_R$ spanned by $f-{}^x f$, where $f \in H_R$ and $x \in G$. 
\end{proposition}

\begin{remark}
The untwisted case is stated in \cite[Proof of Lemma 3.1]{Ka}. 
\end{remark}

\begin{proof}
We first show that $f-{}^x f \in [H_R, H_R]_{\th, \omega}$. 

Let $\CK$ be a $\th$-stable open compact pro-$p$ subgroup $\CK$ of $G$ such that $f$ is left $\CK$-invariant and $\omega(\CK)=1$. We may write $f$ as a linear combination of $\mathbbm{1}_{\CK g}$ for $g \in G$. We have ${}^x \mathbbm{1}_{\CK g}=\omega(x) \mathbbm{1}_{x \CK g \th(x) \i}$. By definition,
\begin{align*}
\omega(x) \mathbbm{1}_{x \CK g \th(x) \i} &=\frac{\omega(x)}{\mu_G(\CK)} \mathbbm{1}_{x \CK} \mathbbm{1}_{\CK g \th(x) \i} \equiv \frac{1}{\mu_G(\CK)} \mathbbm{1}_{\CK g \th(x) \i} \mathbbm{1}_{\th(x) \CK}  \\ 
&=\frac{\mu_G(\CK)}{\mu_G(\CK g \CK)} \mathbbm{1}_{\CK g \CK} \mod [H_R, H_R]_{\th, \omega}.
\end{align*}

In particular, by taking $x=1$, we have $\mathbbm{1}_{\CK g} \equiv \frac{\mu_G(\CK)}{\mu_G(\CK g \CK)} \mathbbm{1}_{\CK g \CK} \mod [H_R, H_R]_{\th, \omega}.$ Hence $\mathbbm{1}_{\CK g} \equiv {}^x \mathbbm{1}_{\CK g} \mod [H_R, H_R]_{\th, \omega}$. 

Now let $X, X'$ be open subsets of $G$ such that $\omega \mid_X$ is constant. We show that $[\mathbbm{1}_X, \mathbbm{1}_{X'}]_{\th, \omega}$ lies in the span of $f-{}^x f$. 

We choose a $\th$-stable open compact pro-$p$ subgroup $\CK$ such that $X$ is left $\CK$-invariant, $X'$ is right $\CK$-invariant and $\omega(\CK)=1$. We may write $\mathbbm{1}_X$ as a sum of $\mathbbm{1}_{\CK g}$ and write $\mathbbm{1}_{X'}$ as a sum of $\mathbbm{1}_{g' \CK}$. Then $[\mathbbm{1}_X, \mathbbm{1}_{X'}]_{\th, \omega}=\sum_{g, g'} [\mathbbm{1}_{\CK g}, \mathbbm{1}_{g' \CK}]_{\th, \omega}$. We have 
\begin{align*}
[\mathbbm{1}_{\CK g}, \mathbbm{1}_{g' \CK}]_{\th, \omega} &=\mathbbm{1}_{\CK g} \mathbbm{1}_{g' \CK}-\omega(g) \i \mathbbm{1}_{g' \CK} \mathbbm{1}_{\CK \th(g)} \\&=\frac{\mu_G(\CK)^2}{\mu_G(\CK g g' \CK)} \mathbbm{1}_{\CK g g' \CK}-\omega(g) \i \mu_G(\CK) \mathbbm{1}_{g' \CK \th(g)} \\ &=\frac{\mu_G(\CK)^2}{\mu_G(\CK g g' \CK)} (\mathbbm{1}_{\CK g g' \CK}-\frac{\omega(g) \i \mu_G(\CK g g' \CK)}{\mu_G(\CK)} \mathbbm{1}_{g' \CK \th(g)}). 
\end{align*}

Set $n=\frac{\mu_G(\CK g g' \CK)}{\mu_G(\CK)}$. We have that $\mathbbm{1}_{\CK g g' \CK}=\sum_{i=1}^n \mathbbm{1}_{k_i g g' \CK}$ for some $k_i \in \CK$. Note that $(k_i g) \i (k_i g g' \CK) \th(k_i g)=g' \CK \th(k_i) \th(g)=g' \CK \th(g)$ and $\omega(k_i g)=\omega(g)$. Thus $\omega(g) \i \mathbbm{1}_{g' \CK \th(g)}={}^{(k_i g) \i} \mathbbm{1}_{k_i g g' \CK}$ and $$[\mathbbm{1}_{\CK g}, \mathbbm{1}_{g' \CK}]_{\th, \omega}=\frac{\mu_G(\CK)^2}{\mu_G(\CK g g' \CK)} \sum_{i=1}^{n} (\mathbbm{1}_{k_i g g' \CK}-{}^{(k_i g) \i} \mathbbm{1}_{k_i g g' \CK}).$$

The proposition is proved.
\end{proof}

\subsection{} Let $\bar H_R=H_R/[H_R, H_R]_{\th, \omega}$. This is the $(\th, \omega)$-twisted cocenter of $H_R$. The distributions on $G$ are the $R$-valued linear functions on $H_R$. We say that a distribution is $(\th, \omega)$-invariant if it vanishes on $f-{}^x f$ for all $f \in H_R$ and $x \in G$. In other words, the twisted action of $G$ on $H_R$ induces a twisted action of $G$ on the set of distributions via ${}^x j(f)=j({}^{x \i} f)$ for any distribution $j$, $f \in H_R$ and $x \in G$. A distribution $j$ is $(\th, \omega)$-invariant if and only if $j={}^x j$ for any $x \in G$. We denote by $J(G)=\bar H_R^*$ the set of all $(\th, \omega)$-invariant distributions on $G$. 

\section{Newton decomposition of $G$}

\subsection{}\label{k-n} The $\th$-twisted conjugation action on $\tW$ is defined by $w \cdot_\th w'=w w' \th(w) \i$. Let $\text{cl}_{\th}(\tW)$ be the set of $\th$-twisted conjugacy classes of $\tW$. Since the action of $\th$ on $V$ is of finite order. Each $\th$-orbit on $\Omega$ is a finite set. Therefore

(a) Each $\th$-twisted conjugacy class of $\tilde W$ intersects only finitely many $W_a$ cosets. 

Following \cite{HN1}, we define two arithmetic invariants on $\text{cl}_{\th}(\tW)$. 

Note that each $\th$-twisted conjugacy class of $\tW$ lies in a single $\th$-orbit on the cosets $\tW/W_a \cong \Omega$. Let $\Omega_{\th}$ be the set of $\th$-coinvariants of $\Omega$. The projection map $\tW \to \Omega_\th$ factors through $\text{cl}_{\th}(\tW)$. The induced map $\k: \text{cl}_{\th}(\tW) \to \Omega_\th$ gives one invariant. 

Recall that $\tW=X_*(Z)_{\Gal(\bar F/F)} \rtimes W_0$. We regard $\th$ as an element in the group $\tW \rtimes \<\th\>$ and we extend the length function $\ell$ on $\tW$ to $\tW \rtimes \<\th\>$ by requiring that $\ell(\th)=0$. For any $w \in \tW$, $(w \th)^{m |W_0|} \in X_*(Z)_{\Gal(\bar F/F)},$ where $m$ is the order of the automorphism $\th$ on $\tW$ and $|W_0|$ is the order of the relative Weyl group. 

For $w \in \tW$, we set $\nu_w=\l/n \in V$, where $n$ is a positive integer and $\l \in X_*(Z)_{\Gal(\bar F/F)}$ with $(w \th)^n=t^\l$. It is easy to see that $\nu_w$ is independent of the choice of the power $n$. We $\nu_w$ the Newton point of $w$. Let $\bar \nu_w$ be the unique dominant element in the $W_0$-orbit of $\nu_w$. The map $w \mapsto \bar \nu_w$ is constant on each conjugacy class of $\tW$. This gives another invariant. 

Let $V_+$ be the set of dominant elements in $V$. Set $\aleph=\Omega_\th \times V_+$. We have a map $$\pi=(\k, \bar \nu): \text{cl}_{\th}(\tW) \to \aleph.$$

\subsection{} Let $\tW_{\min}$ be the subset of $\tW$ consisting of elements of minimal length in their $\th$-twisted conjugacy classes of $\tW$. For any $\nu=(\t, v) \in \aleph$, we set $$X_\nu=\cup_{w \in \tW_{\min}; \pi(w)=\nu} \CI \dot w \CI \quad \text{ and } \quad G(\nu)=G \cdot_\th X_\nu.$$ Here $\cdot_\th$ means the $\th$-twisted conjugation action of $G$ defined by $g \cdot_\th g'=g g' \th(g) \i$. We call $G(\nu)$ the {\it Newton stratum} of $G$ corresponding to $\nu$. 

The main result of this section is 

\begin{theorem}\label{newton-gf}
We have the Newton decomposition $$G=\bigsqcup_{\nu \in \aleph} G(\nu).$$
\end{theorem}

\

The proof is based on some remarkable combinatorial properties of the minimal length elements of $\tW$ established in \cite{HN1} and the reduction method in \cite{He-14}. 

\subsection{} We follow \cite{HN1}. For $w, w' \in \tW$ and $s \in \tilde \BS$, we write $w \xrightarrow{s}_\th w'$ if $w'=s w \th(s)$ and $\ell(w') \le \ell(w)$.  We write $w \to_\th w'$ if there is a sequence $w=w_0, w_1, \cdots, w_n=w'$ of elements in $\tW$ such that for any $k$, $w_{k-1} \xrightarrow{s}_\th w_k$ for some $s \in \tilde \BS$. We write $w \approx_\th w'$ if $w \to_\th w'$ and $w' \to_\th w$. It is easy to see that if $w \to_\th w'$ and $\ell(w)=\ell(w')$, then $w \approx_\th w'$. 

The following result is proved in \cite[Theorem A]{HN1}.

\begin{theorem}\label{min}
Let $w \in \tW$. Then there exists an element $w' \in \tW_{\min}$ with $w \to_\th w'$. 
\end{theorem}

\subsection{}\label{DL-red} Now we recall the reduction method in \cite{He-14}. 

Let $w \in \tW$ and $s \in \tilde \BS$. We have an explicit formula on the multiplication of Bruhat cells \[\CI \dot s \CI \dot w \CI=\begin{cases} \CI \dot s \dot w \CI, & \text{ if } s w>w; \\ \CI \dot s \dot w \CI \sqcup \CI \dot w \CI, & \text{ if } s w<w.\end{cases}\] \[\CI \dot w \CI \dot s \CI=\begin{cases} \CI \dot w \dot s \CI, & \text{ if } w s>w; \\ \CI \dot w \dot s \CI \sqcup \CI \dot w \CI, & \text{ if } w s<w.\end{cases}\]

We have the following simple but very useful properties: 

\begin{enumerate}

\item $G \cdot_\th \CI \dot w \CI=G \cdot_\th \CI \dot w' \CI$ if $w \approx_\th w'$; 

\item $G \cdot_\th \CI \dot w \CI=G \cdot_\th \CI \dot s \dot w \CI \cup G \cdot_\th \CI \dot s \dot w \th(\dot s) \CI$ for $s \in \tilde \BS$ with $s w s<w$. 
\end{enumerate}

\begin{proposition}\label{fini}
Let $X$ be a compact subset of $G$. Then there exists a finite subset $\{\nu_1, \cdots, \nu_k\}$ of $\aleph$ such that $$X \subset \cup_i G(\nu_i).$$
\end{proposition}

\begin{proof}
Since any compact subset of $G$ is contained in a finite union of $\CI$-double cosets, it suffices to prove the statement for $\CI \dot w \CI$ for any $w \in \tW$. 

We argue by induction on $\ell(w)$. 

If $w \in \tW_{\min}$, the statement is obvious. If $w \notin \tW_{\min}$, then by Theorem \ref{min}, there exists $w' \in \tW$ and $s \in \tilde \BS$ such that $w \approx_\th w'$ and $s w' \th(s)<w'$. Then by \S \ref{DL-red}, $$G \cdot_\th \CI \dot w \CI=G \cdot_\th \CI \dot w' \CI \subset G \cdot_\th \CI \dot s \dot w' \th(\dot s) \CI \cup G \cdot_\th \CI \dot s \dot w' \CI.$$ Note that $\ell(s w'), \ell(s w' \th(s))<\ell(w)$, the statement for $w$ follows from inductive hypothesis on $s w'$ and on $s w' \th(s)$. 
\end{proof}

\subsection{} Since $G=\sqcup_{w \in \tW} \CI \dot w \CI$, by Proposition \ref{fini}, $G=\cup_{\nu \in \aleph} G(\nu)$. In order to show that $\cup_{\nu \in \aleph} G(\nu)$ is a disjoint union, we use some properties on the straight conjugacy classes of $\tW$. 

By definition, an element $w \in \tW$ is {\it $\th$-straight} if $\ell((w \th)^k)=k \ell(w)$ for all $k \in \BN$. A $\th$-twisted conjugacy class is {\it straight} if it contains a $\th$-straight element. It is easy to see that the $\th$-straight elements in a given straight $\th$-twisted conjugacy class $\CO$ are exactly the minimal length elements in $\CO$. 

Let $\text{cl}_{\th}(\tW)_{str}$ be the set of straight conjugacy classes. It is proved in \cite[Theorem 3.3]{HN1} that 

\begin{theorem}
The map $\pi: \text{cl}_{\th}(\tW) \to \aleph$ induces a bijection between $\text{cl}_{\th}(\tW)_{str}$ and $Im(\pi)$. 
\end{theorem}

In other words, we have a well-defined map $\text{cl}_{\th}(\tW) \to \text{cl}_{\th}(\tW)_{str}$ which sends a conjugacy class $\CO$ of $\tW$ to the unique straight conjugacy class in $\pi \i(\pi(\CO))$. It is proved in \cite[Proposition 2.7]{HN1} that this map is ``compatible'' with the length function in the following sense. 

\begin{theorem}\label{str}
Let $\CO \in \text{cl}_{\th}(\tW)$ and $\CO'$ be the associated straight conjugacy class. Then for any $w \in \CO$, there exists a triple $(x, K, u)$ with $w \to_\th u x$, where $x$ is a straight element in $\CO'$, $K$ is a subset of $\tilde \BS$ such that $W_K$ is finite, $x \in {}^K \tilde W$ and $\Ad(x)\th(K)=K$, and $u \in W_K$. 
\end{theorem}

\begin{remark}
We call $(x, K, u)$ a {\it standard triple} associated to $w$. By Theorem \ref{str} and \S\ref{DL-red} (1), we have the following alternative definition of Newton stratum \[G(\nu)=\cup_{(x, K, u) \text{ is a standard triple}; u x \in \tW_{\min}, \pi(x)=\nu} G \cdot \CI \dot u \dot x \CI.\]

\end{remark}

\begin{corollary}\label{finite-w}
For any $\nu \in \aleph$, there are only finitely many $w \in \tW_{\min}$ with $\pi(w)=\nu$. In particular, each fiber of the map $\pi: \text{cl}_{\th}(\tW) \to \aleph$ is finite. 
\end{corollary}

\begin{proof}
Let $\nu \in Im(\pi)$ and $\CO'$ be the associated straight conjugacy class. Let $l$ be the length of any straight element in $\CO'$. Note that there are only finitely many $K \subset \tilde \BS$. In particular, 

(a) $\max\{\ell(u); u \in W_K \text{ for some } K \subset \tilde \BS \text{ with } W_K \text{ finite}\}$ is finite. 

We denote this number by $n$. Then by Theorem \ref{str}, for any $w \in \tW_{\min}$ with $\pi(w)=\nu$, $\ell(w) \le l+n$. By \ref{k-n} (a), $\CO'$ intersects only finitely many $W_a$ cosets and hence any element $w \in \tW_{\min}$ with $\pi(w)=\nu$ is contained in one of those cosets. Since in a given coset of $W_a$, there are only finitely many elements of a given length, there are only finitely many such $w$. 
\end{proof}

\subsection{Proof of Theorem \ref{newton-gf}}\label{2.5} We have shown that $G=\cup_{\nu \in \aleph} G(\nu)$. It remains to show that it is a disjoint union. 

Let $\nu_1 \neq \nu_2$ be elements in $\aleph$. It is easy to see that if $\nu_1$ and $\nu_2$ have different $\Omega_\th$-factor, then $G(\nu_1) \cap G(\nu_2)=\emptyset$. Now assume that $\nu_1$ and $\nu_2$ have the same $\Omega_\th$-factor. Then they have different $V$-factor. 

The remaining part of the proof is a bit technical and we first explain the main idea. Suppose that $G(\nu_1) \cap G(\nu_2) \neq \emptyset$. Then there exists an element $g_1 \in G(\nu_1)$ and $g_2 \in G(\nu_2)$, and an element $g \in G$ that (twisted) conjugate $g_1$ to $g_2$. Then $g$ also (twisted) conjugate a (twisted) power of $g_1$ to a (twisted) power of $g_2$. But as the Newton factor of $\nu_1$ and $\nu_2$ are different, the ``difference'' between the (twisted) $n$th powers of $g_1$ and $g_2$ goes to ``infinity'' as $n$ goes to infinity, and thus can not be (twisted) conjugated by a fixed element $g \in G$.

Now we go back to the proof. Suppose that $G(\nu_1) \cap G(\nu_2) \neq \emptyset$. By definition, there exists $w_1, w_2 \in \tW_{\min}$ with $\pi(w_i)=\nu_i$ and that $G \cdot_\th \CI \dot w_1 \CI \cap G \cdot_\th \CI \dot w_2 \CI \neq \emptyset$. By Theorem \ref{str}, there exists standard triples $(x_i, K_i, u_i)$ associated to $w_i$ for $i=1, 2$. Since $w_i \in \tW_{\min}$, $w_i \approx_\th u_i x_i$. By \S\ref{DL-red} (1), $G \cdot_\th \CI \dot w_i \CI=G \cdot_\th \CI \dot u_i \dot x_i \CI$. Hence
$$G \cdot_\th \CI \dot u_1 \dot x_1 \CI \cap G \cdot_\th \CI \dot u_2 \dot x_2 \CI \neq \emptyset.$$

Let $h_i \in \CI \dot u_i \dot x_i \CI$ and $g \in G$ with $g h_1 \th(g) \i=h_2$. By our assumption on $x_i$ and $K_i$, we have for any $n \in \BN$, 
\begin{align*}\tag{a}(\CI \dot u_i \dot x_i \CI) \th(\CI \dot u_i \dot x_i \CI) \cdots \th^{n-1}(\CI \dot u_i \dot x_i \CI) & \subset (\cup_{w \in W_{K_i}} \CI \dot w \CI) (\CI \dot x_i \th(\dot x_i) \cdots \th^{n-1}(\dot x_i) \CI) \\ &=(\CI \dot x_i \th(\dot x_i) \cdots \th^{n-1}(\dot x_i) \CI) (\cup_{w \in W_{K_i}} \CI \dot w \CI).\end{align*}

Thus $h_i \th(h_i) \cdots \th^{n-1}(h_i) \in  (\CI \dot x_i \th(\dot x_i) \cdots \th^{n-1}(\dot x_i) \CI) (\cup_{w \in W_{K_i}} \CI \dot w \CI)$. 

Let $n_0$ be a positive integer such that $(x_i \th)^{n_0}=t^{\l_i} \in \tW \rtimes \<\th\>$ for some $\l_i \in X_*(Z)_{\Gal(\bar F/F)}$. Since the $V$-factor of $\nu_1$ and $\nu_2$ are different, $\l_1$ is not in the $W_0$-orbit of $\l_2$. For any $l \in \BN$, we have $$g h_1 \th(h_1) \cdots \th^{n_0 l-1}(h_1) \th^{n_0 l}(g) \i=h_2 \th(h_2) \cdots \th^{n_0 l-1}(h_2).$$ Hence $$g (\CI t^{l \l_1} \CI) (\cup_{w \in W_{K_1}} \CI \dot w \CI) \th^{n_0 l}(g) \i \cap (\CI t^{l \l_2} \CI) (\cup_{w \in W_{K_2}} \CI \dot w \CI) \neq \emptyset$$ and thus $$g (\CI t^{l \l_1} \CI) (\cup_{w \in W_{K_1}} \CI \dot w \CI) \th^{n_0 l}(g) \i (\cup_{w \in W_{K_2}} \CI \dot w \CI) \cap (\CI t^{l \l_2} \CI) \neq \emptyset.$$

We have that $g \in \CI \dot z \CI$ for some $z \in \tW$. Then for any $l \in \BN$, $$(\CI \dot z \CI) (\CI t^{l \l_1} \CI) (\cup_{w \in W_{K_1}} \CI \dot w \CI) (\CI \th^{n_0 l}(\dot z) \i \CI) (\cup_{w \in W_{K_2}} \CI \dot w \CI) \cap \CI t^{l \l_2} \CI  \neq \emptyset.$$

Let $N_0=\max_{w \in W_{K_1}} \ell(w)+\ell(z)+\max_{w \in W_{K_2}} \ell(w)$. Then $$\CI \dot z \CI, (\cup_{w \in W_{K_1}} \CI \dot w \CI) (\CI \th^{nl}(\dot z) \i \CI) (\cup_{w \in W_{K_2}} \CI \dot w \CI) \subset \cup_{y \in W_a; \ell(y) \le N_0} \CI \dot y \CI$$ and \begin{align*} (\CI \dot z \CI) (\CI t^{l \l_1} \CI) & (\cup_{w \in W_{K_1}} \CI \dot w \CI) (\CI \th^{nl}(\dot z) \i \CI) (\cup_{w \in W_{K_2}} \CI \dot w \CI)  \\ & \subset \cup_{y, y' \in W_a; \ell(y), \ell(y') \le N_0} \CI \dot y t^{l \l_1} \dot y' \CI.\end{align*} In particular, for $l=2 N_0+2 \sharp(W_0)+1$, there exists $y, y' \in W_a$ with $\ell(y), \ell(y') \le N_0$ such that $t^{l \l_2}=y t^{l \l_1} y'$. Assume that $y=y_0 t^{\chi}$ and $y'=t^{\chi'} y'_0$ for $y_0, y'_0 \in W_0$ and $\chi, \chi' \in X_*(Z)_{\Gal(\bar F/F)}$. Then $y_0 t^{\chi+l \l_1+\chi'} y'_0=t^{l \l_2}$. Hence $l \l_2=y_0(\chi+l \l_1+\chi')$ and $l(\l_2-y_0 \l_1)=y_0(\chi+\chi')$. Notice that $\l_2-y_0 \l_1 \neq 0$. Thus \begin{align*}
\ell(t^{y_0(\chi+\chi')}) &  =\ell(t^{\chi+\chi'}) \le \ell(t^{\chi})+\ell(t^{\chi'}) \le \ell(y)+\ell(y_0)+\ell(y')+\ell(y'_0) \\ & \le 2 N_0+2 \sharp(W_0)<l \le \ell(t^{l(\l_2-y_0 \l_1)}).\end{align*} That is a contradiction. \qed

\section{Newton decompositions of $H_R$ and $\bar H_R$}

\subsection{} Recall that $G=\sqcup_{\nu \in \aleph} G(\nu)$. For $\nu \in \aleph$, let $H_R(\nu)$ be the $R$-submodule of $H_R$ consisting of functions supported in $G(\nu)$ and let $\bar H_R(\nu)$ be the image of $H_R(\nu)$ in the cocenter $\bar H_R$. We first establish the Newton decompositions of $H_R$ and $\bar H_R$. 

\begin{theorem}\label{newton-h}
We have that 

(1) $H_R=\bigoplus_{\nu \in \aleph} H_R(\nu)$. 

(2) $\bar H_R=\bigoplus_{\nu \in \aleph} \bar H_R(\nu)$. 
\end{theorem}

\subsection{}\label{adm} A key ingredient in the proof is the admissibility of Newton strata.

Following Grothendieck, a subset $X$ of $G$ is called {\it admissible} if for any open compact subset $C$ of $G$, $X \cap C$ is stable under the right multiplication of an open compact subgroup of $G$. We show that 

(a) For any $\nu \in \aleph$, $G(\nu)$ is admissible. 

In \cite[Theorem A.1]{He-KR}, we show that each Frobenius-twisted conjugacy class of a loop group is admissible. The argument works for the Newton strata of $G$ as well. 

Another, and probably simpler argument to prove (a) is pointed out to me by Julee Kim. Note that each $G(\nu)$ is open and closed. Thus for any open compact subset $C$, $G(\nu) \cap C$ is open and compact. As $G(\nu) \cap C$ is open, for any $g \in G(\nu) \cap C$, there exists an open compact subgroup $\CK$ such that $g \CK \subset G(\nu) \cap C$. As $G(\nu) \cap C$ is compact, there exists finitely many elements $g_i$ and open compact subgroups $\CK_i$ such that $G(\nu) \cap C=\cup_{i} g_i \CK_i$. Set $\CK=\cap_i \CK_i$. Then $G(\nu) \cap C$ is stable under the right multiplication of $\CK$. 

\subsection{Proof of Theorem \ref{newton-h}}
(1) Let $\CK$ be an open compact subgroup of $G$ and $f \in H_R(G, \CK)$. Let $X$ be the support of $f$. Then $X$ is a compact subset of $G$. By Proposition \ref{fini}, there exists a finite subset $\{\nu_1, \cdots, \nu_k\}$ of $\aleph$ such that $X=\sqcup_i (X \cap G(\nu_i))$. By \S \ref{adm} (a), there exists an open compact subgroup $\CK'$ of $\CK$ such that for $1 \le i \le k$, $X \cap G(\nu_i)$ is stable under the right multiplication of $\CK'$. Set $f_i=f \mid_{X \cap G(\nu_i)}$. Since $f$ is invariant under the right action of $\CK$, $f_i$ is invariant under the right action of $\CK'$. Moreover, the support of $f_i$ is $X \cap G(\nu_i)$, which is a compact subset of $G$. Thus $f_i \in H_R(\nu_i)$. So $H_R=\sum_{\nu \in \aleph} H_R(\nu)$. 

On the other hand, suppose that $f=\sum_\nu f_\nu$, where $f_\nu \in H_R(\nu)$ and only finitely many $f_\nu$'s are nonzero. Since $G=\sqcup_\nu G(\nu)$, $f_\nu=f \mid_{G(\nu)}$. In particular, $f_\nu$ is determined by $f$. Thus the decomposition $H_R=\sum_{\nu \in \aleph} H_R(\nu)$ is a direct sum decomposition. 

(2) Let $f \in H_R$ and $x \in G$. By (1), we may write $f$ as $f=\sum_{\nu} f_{\nu}$, where $f_{\nu} \in H_R(\nu)$. By definition, ${}^x f_{\nu} \in H_R(\nu)$. Thus $f-{}^x f=\sum_{\nu} (f_{\nu}-{}^x f_{\nu}) \in \oplus_{\nu} ([H_R, H_R]_{\th, \omega} \cap H_R(\nu))$. The direct sum decomposition $\bar H_R=\oplus_{\nu} \bar H_R(\nu)$ follows from $H_R=\oplus_{\nu} H_R(\nu)$ and $[H_R, H_R]_{\th, \omega}=\oplus_{\nu} ([H_R, H_R]_{\th, \omega} \cap H_R(\nu))$. \qed

\section{Newton decomposition and IM type generators of $\bar H_R(G, \CI_n)$}\label{4}

\subsection{} Let $\CK$ be an open compact subgroup of $G$. Recall that $H_R(G, \CK)$ is the Hecke algebra of compactly supported, $\CK \times \CK$-invariant functions on $G$. For any $\nu \in \aleph$, we denote by $H_R(G, \CK; \nu)$ the $R$-submodule of $H_R(G, \CK)$ consisting of functions supported in $G(\nu)$. 

Note that $G(\nu)$ is not stable under the right action of $\CK$. In other words, there exists a $\CK\times\CK$-orbit $X$ that intersects at least two Newton strata. By Theorem \ref{newton-h} (1), $\mathbbm{1}_X=\sum_{\nu; G(\nu) \cap X \neq \emptyset} \mathbbm{1}_{X \cap G(\nu)}$ with $\mathbbm{1}_{X \cap G(\nu)} \in H_R(\nu)$. As $X \cap G(\nu) \neq X$ for any $\nu$, we see that $\mathbbm{1}_{X \cap G(\nu)} \notin H_R(G, \CK; \nu)$. Thus $$H_R(G, \CK) \supsetneqq \oplus_{\nu} H_R(G, \CK; \nu).$$

\subsection{}\label{4.2} Let $\CI_n$ be the $n$-th Moy-Prasad subgroup associated to the barycenter of $\fka_C$. Since the $(\CI_n)_n$ form a fundamental system of open compact subgroups of $G$, we have $H_R=\varinjlim H_R(G, \CI_n)$. 

For any $w \in \tW$, let $H_w$ be the $R$-submodule of $H_R$ consisting of locally constant functions supported in $\CI \dot w \CI$. For any $w \in \tW$ and $n \in \BN$, set $H_R(G, \CI_n)_w=H_R(G, \CI_n) \cap H_w$. 

Let $\bar H_R(G, \CI_n)$, $\bar H_R(G, \CI_n; \nu)$ and $\bar H_R(G, \CI_n)_w$ be the image of $H_R(G, \CI_n)$, $H_R(G, \CI_n; \nu)$ and $H_R(G, \CI_n)_w$ in $\bar H_R$, respectively. The main results of this section are the Newton decomposition and the Iwahori-Matsumoto type generators of $\bar H_R(G, \CI_n)$.  

\begin{theorem}\label{newton-hn}
Let $n \in \BN$ with $\omega(\CI_{n-1})=1$. Then 

(1) We have $\bar H_R(G, \CI_n)=\bigoplus_{\nu \in \aleph} \bar H_R(G, \CI_n; \nu)$. 

(2) For any $\nu \in \aleph$, we have $\bar H_R(G, \CI_n; \nu)=\sum_{w \in \tW_{\min}; \pi(w)=\nu} \bar H_R(G, \CI_n)_w$. 
\end{theorem}

\smallskip

Since $H_R=\varinjlim H_R(G, \CI_n)$, as a consequence of Theorem \ref{newton-hn} (2), we have the Iwahori-Matsumoto type generators of $\bar H_R$. 

\begin{corollary}\label{IM-h}
Let $\nu \in \aleph$. Then $$\bar H_R(\nu)=\sum_{w \in \tW_{\min}; \pi(w)=\nu} \bar H_w,$$ where $\bar H_w$ is the image of $H_w$ in $\bar H_R$. 
\end{corollary}

We first establish the following multiplication formula.

\begin{proposition}\label{mult}
Let $w, w' \in \tW$ with $\ell(w)+\ell(w')=\ell(w w')$. Then for any $g \in \CI \dot w \CI$ and $g' \in \CI \dot w' \CI$, $$\mathbbm{1}_{\CI_n g \CI_n} \mathbbm{1}_{\CI_n g' \CI_n}=\mu_G (\CI_n) \mathbbm{1}_{\CI_n g g' \CI_n}.$$
\end{proposition}

\begin{remark}
This formula was known for $GL_n$ by Howe \cite{howe2} and for split groups by Ganapathy \cite{Ga}. 
\end{remark}

\subsection{} Recall that $\breve F$ is the completion of the maximal unramified extension of $F$ with valuation ring $\CO_{\breve F}$ and residue field $\bar \kk$ and $\s$ is the Frobenius morphism of $\breve F$ over $F$. Set $\breve G=\BG(\breve F)$. We denote the Frobenius morphism on $\breve G$ again by $\s$. Then $G=\breve G^\s$. In order to prove Proposition \ref{mult}, we use some facts on $\breve G$. 

Let $S$ be a maximal $\breve F$-split torus of $\BG$ which is defined over $F$ and contains $A$. Denote by $\breve \sA$ the apartment corresponding to $S$ over $\breve F$. By \cite[5.1.20]{BT2}, we have a natural isomorphism $\sA \cong \breve \sA^\s$. Denote by $\breve \fka_C$ the unique $\s$-invariant facet of $\breve \sA$ containing $\fka_C$ and denote by $\breve \CI$ the associated Iwahori subgroup over $\breve F$. Then $\CI=\breve \CI \,^\s$. Let $\breve \tW$ be the Iwahori-Weyl group over $\breve F$ and $\breve W_a$ be the associated affine Weyl group. Let $\breve \ell$ be the corresponding length function on $\breve \tW$.

We have a natural isomorphism $\tW \cong \breve \tW \,^{\s}$. It is proved in \cite[Proposition 1.11 \& sublemma 1.12]{Ri} that for $w, w' \in \tW$, $\breve \ell(w w')=\breve \ell(w)+\breve \ell(w')$ if $\ell(w w')=\ell(w)+\ell(w')$. 

\begin{lemma}\label{fix}
Let $g \in G$. Then $(\breve \CI_n g \breve \CI_n/\breve \CI_n)^\s=\CI_n g \CI_n/\CI_n$.
\end{lemma}

\begin{proof}
We identify $\breve \CI_n g \breve \CI_n/\breve \CI_n$ with $\breve \CI_n/(\breve \CI_n \cap g \breve \CI_n g \i)$ and $\CI_n g \CI_n/\CI_n$ with $\CI_n/(\CI_n \cap g \CI_n g \i)$. The natural map $\CI_n \to (\breve \CI_n g \breve \CI_n/\breve \CI_n)^\s$ has kernel $\CI_n \cap g \CI_n g \i$ and induces an injective map $$\CI_n/(\CI_n \cap g \CI_n g \i) \to (\breve \CI_n/(\breve \CI_n \cap g \breve \CI_n g \i))^\s.$$ 

Note that $\breve \CI_n \cap g \breve \CI_n g \i$ is a pro-$p$ subgroup of $\breve \CI_n$. Thus any element in $(\breve \CI_n/(\breve \CI_n \cap g \breve \CI_n g \i))^\s$ can be lifted to an element in $\breve \CI_n^\s=\CI_n$. So the map $\CI_n/(\CI_n \cap g \CI_n g \i) \to (\breve \CI_n/(\breve \CI_n \cap g \breve \CI_n g \i))^\s$ is also surjective.
\end{proof}

\begin{lemma}\label{qq}
Let $w \in \tW$ and $g \in \CI \dot w \CI$. Then $\sharp \CI_n g \CI_n/\CI_n=q^{\breve \ell(w)}$. 
\end{lemma}

\begin{proof}
Note that for any $\tau \in \Omega$ and $w \in \tW$, $\CI_n \dot \tau \dot w \CI_n=\dot \tau \CI_n \dot w \CI_n$ and $\sharp \CI_n \dot \tau \dot w \CI_n/\CI_n=\sharp \CI_n \dot w \CI_n/\CI_n$. Thus it suffices to consider the case where $w \in W_a$. 

We fix a reduced expression $w=s_1 s_2 \cdots s_k$ with $s_i \in \tilde \BS$ for all $i$. Set $g'=\dot s_1 \dot s_2 \cdots \dot s_k$. Suppose that $g=h g' h'$ for $h, h' \in \CI$. Then $\CI_n g \CI_n=h \CI_n g' \CI_n h'$. It suffices to prove the statement for $g'$. 

Let $\breve \sR$ be the set of affine roots of $\breve G$. We define $\breve \sR(w)=\{\a \in \breve \sR; \a>0, w (\a)<0\}$. By \cite[Sublemma 1.12]{Ri}, $\breve \sR(w)=\breve \sR(s) \sqcup s \breve \sR(w s)$ for any $s \in \tilde \BS$ with $w s<w$. Define the action of $\breve \CI^k$ on $\breve \CI \dot s_1 \breve \CI \times \cdots \times \breve \CI \dot s_k \breve \CI$ by $(i_1, \cdots, i_k) \cdot (g_1, \cdots g_k)=(g_1 i_1^{-1}, i_1 g_1 i_2^{-1}, \cdots, i_{k-1} g_k i_k \i)$. Let $\breve \CI \dot s_1 \breve \CI \times_{\breve \CI} \cdots \times_{\breve \CI} \breve \CI \dot s_k \breve \CI/\breve \CI$ be the quotient space. By standard facts on Tits systems (see \cite[Theorem 5.1.3 (i)]{Ku}), the multiplication map $$\breve \CI \dot s_1 \breve \CI \times_{\breve \CI} \cdots \times_{\breve \CI} \breve \CI \dot s_k \breve \CI/\breve \CI \to \breve \CI g' \breve \CI/\breve \CI$$ is bijective. 

Similarly, the map $$\breve \CI_n \dot s_1 \breve \CI_n \times_{\breve \CI_n} \cdots \times_{\breve \CI_n} \breve \CI_n \dot s_k \breve \CI_n/\breve \CI_n \to \breve \CI_n g' \breve \CI_n/\breve \CI_n$$ is bijective. 

Note that the multiplication map is $\s$-equivariant. Thus by Lemma \ref{fix}, \begin{align*} \CI_n g' \CI_n/\CI_n & \cong (\breve \CI_n \dot s_1 \breve \CI_n \times_{\breve \CI_n} \cdots \times_{\breve \CI_n} \breve \CI_n \dot s_k \breve \CI_n/\breve \CI_n)^\s \\ & \cong \CI_n \dot s_1 \CI_n \times_{\CI_n} \cdots \times_{\CI_n} \CI_n \dot s_k \CI_n/\CI_n.
\end{align*}
Therefore $\sharp \CI_n g' \CI_n/\CI_n=\sharp(\CI_n \dot s_1 \CI_n/\CI_n) \sharp(\CI_n \dot s_2 \CI_n/\CI_n) \cdots \sharp (\CI_n \dot s_k \CI_n/\CI_n)$. 

It remains to show that for any $s \in \tilde \BS$, $\sharp \CI_n \dot s \CI_n/\CI_n=q^{\breve \ell(s)}$. 

Let $\breve \Phi$ be the set of roots of $\breve G$ relative to $S$.  and $\breve \Phi'$ be the set of non-divisible roots in $\breve \Phi$. Let $T$ be the centralizer of $S$. As in \cite[\S3.1]{Tits}, for any $a \in \breve \Phi'$, there exist $\a_a, \b_a \in \breve \sR$ whose vector parts are $a$ and such that the product mappings
\begin{gather*} 
\Pi_{a \in \breve \Phi'} X_{\a_a} \times T_0 \to \breve \CI, \\
\Pi_{a \in \breve \Phi'} X_{\b_a} \times T_0 \to \breve \CI \cap \dot s\breve \CI \dot s \i 
\end{gather*} 
are bijective (for any ordering of the factors of the product), where $X_{\a}$ is the affine root subgroup corresponding to the affine root $\a$ as in \cite[\S1.2]{Tits} and $T_0$ is the unique parahoric subgroup of $T(\breve F)$. 

By the definition of $\breve \CI_n$, the product mappings 
\begin{gather*} 
\Pi_{a \in \breve \Phi'} X_{\a_a+n} \times T_n \to \breve \CI_n, \\
\Pi_{a \in \breve \Phi'} X_{\b_a+n} \times T_n \to \breve \CI_n \cap \dot s\breve \CI_n \dot s \i 
\end{gather*} 
are also bijective.

Hence $$\breve \CI_n \dot s \breve \CI_n/\breve \CI_n \cong \Pi_{a \in \breve \Phi'} X_{\a_a+n}/\Pi_{a \in \breve \Phi'} X_{\b_a+n} \cong \Pi_{a \in \breve \Phi'} X_{\a_a}/\Pi_{a \in \breve \Phi'} X_{\b_a} \cong \breve \CI \dot s \breve \CI/\breve \CI.$$ 

By \cite[Proof of Proposition 1.11]{Ri}, $\breve \CI \dot s \breve \CI/\breve \CI$ is an affine space over $\bar \kk$ of dimension $\breve \ell(s)$. By Lemma \ref{fix}, $\CI \dot s \CI/\CI=(\breve \CI_n \dot s \breve \CI_n/\breve \CI_n)^\s$ is the set of $\kk$-valued points of an affine space of dimension $\breve \ell(s)$. Therefore, $\sharp \CI_n \dot s \CI_n/\CI_n=q^{\breve \ell(s)}$. 
\end{proof}

\subsection{Proof of Proposition \ref{mult}}
Define the action of $\CI_n$ on $\CI_n g \CI_n \times \CI_n g' \CI_n$ by $h \cdot (z, z')=(z h \i, h z')$. Let $\CI_n g \CI_n \times_{\CI_n} \CI_n g' \CI_n$ be the quotient. We consider the map induced from the multiplication $$\CI_n g \CI_n \times_{\CI_n} \CI_n g' \CI_n/\CI_n \to G/\CI_n.$$ It is obvious that $\CI_n g g' \CI_n/\CI_n$ is contained in the image. By Lemma \ref{qq}, \begin{gather*} \sharp \CI_n g g' \CI_n/\CI_n=q^{\breve \ell(w w')}, \\ \sharp\CI_n g \CI_n \times_{\CI_n} \CI_n g' \CI_n/\CI_n=\sharp\CI_n g \CI_n/\CI_n \cdot \sharp \CI_n g' \CI_n/\CI_n=q^{\breve \ell(w)} q^{\breve \ell(w')}.
\end{gather*} Since $\ell(w w')=\ell(w)+\ell(w')$, $\breve \ell(w w')=\breve \ell(w)+\breve \ell(w')$. Therefore $$\sharp\CI_n g \CI_n \times_{\CI_n} \CI_n g' \CI_n/\CI_n=\sharp \CI_n g g' \CI_n/\CI_n.$$ Thus the image of the map $\CI_n g \CI_n \times_{\CI_n} \CI_n g' \CI_n/\CI_n \to G/\CI_n$ is $\CI_n g g' \CI_n/\CI_n$ and the map is bijective. Now the statement follows from \S\ref{hecke}(a). \qed

\smallskip

Similar to \S\ref{DL-red}, we have the following inductive result.

\begin{lemma}\label{bar-h-w}
Let $n \in \BN$ with $\omega(\CI_{n-1})=1$. Let $w \in \tW$ and $s \in \BS$. 

(1) If $\ell(s w \th(s))=\ell(w)$, then $\bar H_R(G, \CI_n)_w=\bar H_R(G, \CI_n)_{s w \th(s)}$. 

(2) If $s w \th(s)<w$, then $\bar H_R(G, \CI_n)_w \subset \bar H_R(G, \CI_n)_{s w \th(s)}+\bar H_R(G, \CI_n)_{s w}$.  
\end{lemma}

\begin{proof}
Without loss of generality, we may assume that $s w<w$. By definition, $H_R(G, \CI_n)_w$ is spanned by $\mathbbm{1}_{\CI_n g \CI_n}$ with $g \in \CI \dot w \CI$. Since $s w<w$, for any $g \in \CI \dot w \CI$, there exists $g_1 \in \CI \dot s \CI$ and $g_2 \in \CI \dot s \dot w \CI$ with $g=g_1 g_2$. Since $\CI_n g_1 \CI_n \subset g_1 \CI_{n-1}$, we have $\omega \mid_{\CI_n g_1 \CI_n}$ is constant. By Proposition \ref{mult}, 
\begin{align*}
\mathbbm{1}_{\CI_n g \CI_n} &=\frac{1}{\mu_G(\CI_n)} \mathbbm{1}_{\CI_n g_1 \CI_n} \mathbbm{1}_{\CI_n g_2 \CI_n} \\ & \equiv \frac{\omega(g_1) \i}{\mu_G(\CI_n)} \mathbbm{1}_{\CI_n g_2 \CI_n} \mathbbm{1}_{\CI_n \th(g_1) \CI_n}  \mod [H_R(G, \CI_n), H_R(G, \CI_n)]_{\th, \omega}.
\end{align*}

If $\ell(s w \th(s))=\ell(w)$, then $s w \th(s)>s w$ and $\mathbbm{1}_{\CI_n g_2 \CI_n} \mathbbm{1}_{\CI_n g_1 \CI_n} \in H_R(G, \CI_n)_{s w \th(s)}$. Thus $\bar H_R(G, \CI_n)_w \subset \bar H_R(G, \CI_n)_{s w \th(s)}$. Similarly, $\bar H_R(G, \CI_n)_{s w \th(s)} \subset \bar H_R(G, \CI_n)_w$. Part (1) is proved. 

If $s w \th(s)<w$, then $\mathbbm{1}_{\CI_n g_2 \CI_n} \mathbbm{1}_{\CI_n g_1 \CI_n} \subset H_R(G, \CI_n)_{s w \th(s)}+H_R(G, \CI_n)_{s w}$. Part (2) is proved. 
\end{proof}

\subsection{Proof of Theorem \ref{newton-hn}}
By Theorem \ref{newton-h}, $\sum_{\nu} \bar H_R(G, \CI_n; \nu)$ is a direct sum. By definition, for any $w \in \tW_{\min}$, $\bar H_R(G, \CI_n)_w \subset \bar H_R(G, \CI_n; \pi(w))$. Thus it suffices to show that $$\bar H_R(G, \CI_n)=\sum_{w \in \tW_{\min}} \bar H_R(G, \CI_n)_w.$$

Since $H_R(G, \CI_n)=\oplus_{x \in \tW} H_R(G, \CI_n)_{x}$, we argue by induction on $\ell(x)$ that $$\bar H_R(G, \CI_n)_{x} \subset \sum_{w \in \tW_{\min}} \bar H_R(G, \CI_n)_w.$$

If $x \in \tW_{\min}$, then the statement is obvious. 

If $x \notin \tW_{\min}$, then by Theorem \ref{min}, there exists $x' \in \tW$ and $s \in \tilde \BS$ such that $x \approx_\th x'$ and $s x' \th(s)<x'$. Then by Lemma \ref{bar-h-w}, $\bar H_R(G, \CI_n)_{x}=\bar H_R(G, \CI_n)_{x'} \subset \bar H_R(G, \CI_n)_{s x'}+\bar H_R(G, \CI_n)_{s x' \th(s)}$. Since $\ell(s x'), \ell(s x' \th(s))<\ell(x)$, the statement follows from inductive hypothesis on $s x'$ and $s x' \th(s)$. \qed

\subsection{} By Theorem \ref{str} and Lemma \ref{bar-h-w} (1), the Iwahori-Matsumoto type generators of $\bar H_R(G, \CI_n; \nu)$ can be formulated as follows:

(a) Any element in $\bar H_R(G, \CI_n; \nu)$ can be represented by an element in $H_R(G, \CI_n)$ with support in $\cup_{(x, K, u) \text{ is a standard triple}; u x \in \tW_{\min}, \pi(x)=\nu} \CI \dot u \dot x \CI$.

\section{Howe's conjecture}\label{5} 

\subsection{} Let $X$ be a compact subset of $G$. Recall that $$G \cdot_\th X=\{g x \th(g) \i; g \in G, x \in X\}.$$ We denote by $J(G \cdot_\th X)$ the set of $(\th, \omega)$-invariant distributions of $G$ supported in $G \cdot_\th X$. Let $\CK$ be an open compact subgroup of $G$. Howe's conjecture asserts that 

\begin{theorem}\label{howe-1}
The restriction $J(G \cdot_\th X) \mid_{H_R(G, \CK)}$ is finite dimensional.
\end{theorem}

\begin{remark} For ordinary invariant distributions, this is proved by Clozel \cite{Cl}, Barbasch and Moy \cite{BM}. For twisted invariant distributions, this is a new result. Our approach here is different from both \cite{Cl} and \cite{BM}. 
\end{remark}

\subsection{} Let $n \in \BN$ with $\omega(\CI_{n-1})=1$ and $\nu \in \aleph$. Then $\CI_n$ is an open compact subgroup of $G$ and $X_\nu$ is a compact subset of $G$. By definition, the Newton stratum $G(\nu)$ is $G \cdot_\th X_\nu$. 

Recall that $R$ is a commutative $\BZ[\frac{1}{p}]$-algebra. For any $R$-module $M$, we set $M^*=\Hom_R(M, R)$. By the Newton decomposition $G=\sqcup_{\nu \in \aleph} G(\nu)$, we have that $$J(G)=\oplus_{\nu \in \aleph} J(G(\nu)) \text{ and } J(G(\nu))=\bar H_R(\nu)^*.$$We first consider $J(G(\nu)) \mid_{H_R(G, \CI_n)}$ and give an upper bound of its rank. We will show in \S\ref{proof-howe} how the general case in the Howe's conjecture can be reduced to this case. 

\begin{theorem}\label{howe}
Let $\nu \in \aleph$. Then there exists a constant $N_{\nu} \in \BN$ such that for any $n \in \BN$ with $\omega(\CI_{n-1})=1$, $J(G(\nu)) \mid_{H_R(G, \CI_n)}$ is generated by $N_{\nu} [\CI: \CI_n]$ elements. 
\end{theorem}

\begin{proof}
We have that \begin{align*} J(G(\nu)) \mid_{H_R(G, \CI_n)} &=J(G(\nu)) \mid_{\bar H_R(G, \CI_n)}=J(G(\nu)) \mid_{\oplus_{v' \in \aleph} \bar H_R(G, \CI_n; v')} \\ &=J(G(\nu)) \mid_{\bar H_R(G, \CI_n; \nu)}=\bar H_R(G, \CI_n; \nu)^*.\end{align*} Here the first and last equality follow from the definition of invariant distributions, the second equality follows from the Newton decomposition $\bar H_R(G, \CI_n)=\oplus_{\nu \in \aleph} \bar H_R(G, \CI_n; \nu)$ (Theorem \ref{newton-hn} (1)) and the third  equality follows from the fact that $G(\nu) \cap G(\nu')=\emptyset$ for $\nu \neq \nu'$ (Theorem \ref{newton-gf}).

By the Iwahori-Matsumoto type generators of $\bar H_R(G, \CI_n; \nu)$ (Theorem \ref{newton-hn} (2)), we have a surjection $$\oplus_{w \in \tW_{\min}; \pi(w)=\nu} H_R(G, \CI_n)_w \twoheadrightarrow \bar H_R(G, \CI_n; \nu).$$

By Corollary \ref{finite-w}, there are only finitely many $w \in \tW_{\min}$ with $\pi(w)=\nu$. We denote this number by $N_{\nu}$. For each such $w$ and for any $g \in \CI\dot w \CI$, by Lemma \ref{qq}, we have $$\sharp(\CI \dot w \CI/\CI_n)=\sharp(\CI \dot w \CI/\CI) [\CI: \CI_n]=q^{\breve \ell(w)} [\CI: \CI_n], \quad \text{ and }\sharp(\CI_n g \CI_n/\CI_n)=q^{\breve \ell(w)}.$$ Thus $\CI \dot w \CI$ is a union of $[\CI: \CI_n]$ double cosets of $\CI_n$. In particular, $\rank H_R(G, \CI_n)_w=[\CI: \CI_n]$. Therefore $\bar H_R(G, \CI_n; \nu) $ is generated by $\sum_{w \in \tW_{\min}; \pi(w)=\nu} \rank H_R(G, \CI_n)_w=N_{\nu} [\CI: \CI_n]$ elements.
\end{proof}

\subsection{Proof of Theorem \ref{howe-1}}\label{proof-howe} By Proposition \ref{fini}, $X$ is contained in a finite union of Newton strata $G(\nu)$. Therefore $J(G \cdot_\th X)$ is a subset of a finite union of $J(G(\nu))$. For any open compact subgroup $\CK$ of $G$, there exists $n \in \BN$ such that $\omega(\CI_{n-1})=1$ and $\CI_n \subset \CK$. Hence $H_R(G, \CK) \subset H_R(G, \CI_n)$ and $J(G(\nu)) \mid_{H_R(G, \CK)} \subset J(G(\nu)) \mid_{H_R(G, \CI_n)}$. Now the statement follows from Theorem \ref{howe}. \qed

\section{Rigid cocenter}\label{6}

\subsection{} For any $\nu=(\t, v) \in \aleph$, we denote by $M_\nu$ the centralizer of $v$ in $G$, i.e., the subgroup of $G$ generated by $Z(F)$ and the root subgroups $U_a(F)$ for all roots $a$ with $\<v, a\>=0$ (cf. \cite[\S 6.1]{Ko1}). This is a Levi subgroup of $G$. 

We define the rigid and non-rigid part of $G$ by $$G_\rig=\sqcup_{\nu \in \aleph; M_\nu=G} G(\nu), \qquad G_{\text{nrig}}=\sqcup_{\nu \in \aleph; M_\nu \neq G} G(\nu).$$

Let $H_R^\rig$ and $H_R^{\text{nrig}}$ be the subset of $H_R$ consisting of functions supported in $G_\rig$ and $G_{\text{nrig}}$, respectively and let $\bar H_R^\rig$ and $\bar H_R^{\text{nrig}}$ be their images in $\bar H_R$, respectively. We call $\bar H_R^\rig$ the {\it rigid cocenter} and $\bar H_R^{\text{nrig}}$ the {\it non-rigid part of cocenter}. We have $$\bar H_R^\rig=\oplus_{\nu \in \aleph; M_\nu=G} \bar H_R(\nu), \qquad \bar H_R^{\text{nrig}}=\oplus_{\nu \in \aleph; M_\nu \neq G} \bar H_R(\nu).$$

By the Newton decomposition on $\bar H_R$ (Theorem \ref{newton-h}), we have $$\bar H_R=\bar H_R^\rig \oplus \bar H_R^{\text{nrig}}.$$

We denote by $J(G)_\rig$ the set of $(\th, \omega)$-invariant distributions supported in $G_\rig$ and $J(G)_{\text{nrig}}$ the set of $(\th, \omega)$-invariant distributions supported in $G_{\text{nrig}}$. We have $$J(G)_\rig=(\bar H_R^\rig)^*, \qquad J(G)_{\text{nrig}}=(\bar H_R^{\text{nrig}})^* \quad \text{ and } \quad J(G)=J(G)_\rig \oplus J(G)_{\text{nrig}}.$$

The main purpose of this section is to give an explicit description of the rigid cocenter. In a future work \cite{hecke-2}, we will establish the Bernstein-Lusztig type generators of the cocenter, and realize the non-rigid cocenter as a direct sum of $+$-rigid parts of cocenters of proper Levi subgroups.

\subsection{} We first study $J(G)_\rig$. 

Let $\CP$ be a standard parahoric subgroup of $G$, i.e. a parahoric subgroup of $G$ containing $\CI$. Let $\t \in \Omega$. We say that $(\CP, \t)$ is a {\it standard pair} if $\dot \t \th(\CP) \dot \t \i=\CP$. Let $W_{\CP}$ be the (finite) Weyl group of $\CP$. Then the conjugation action of $\t$ induces a length-preserving automorphism on $W_{\CP}$. We denote by $\text{StP}$ the set of all standard pairs.

We have that 

\begin{proposition}\label{rig-g}
The rigid part of $G$ equals $\cup_{(\CP, \t) \in \text{StP}} \, G \cdot_\th \CP \dot \t$. 
\end{proposition}

\begin{proof}
Let $(\CP, \t)$ be a standard pair. Let $n$ be a positive integer such that $(\t \th)^n=t^\l$. Since $\t \in \Omega$, $t^\l \in \Omega$ and thus $\<\l, a\>=0$ for any root $a$ of $G$. Therefore $M_\l=G$ and $M_{\pi(\t)}=G$. Note that $\dot \t \th(\CP) \dot \t \i=\CP$. Similar to the proof of Proposition \ref{fini}, we have $$\CP \dot \t=\cup_w \CP \cdot_\th \CI \dot w \dot \t \CI,$$ where $w \in W_{\CP}$ such that $w \t$ is of minimal length in $\{x w \t \th(x) \i; x \in W_{\CP}\}$. Let $w$ be such an element. Then it is easy to see that $\pi(w \t)=\pi(\t)$. Moreover, for any $x' \in \tW$, we may write $x'$ as $x'=x_1 x_2$ for $x_1 \in \tW^{\CP}$ and $x_2 \in W_{\CP}$. Then $x_2 w \t \th(x_2) \i \in W_{\CP} \t$ and \begin{align*} \ell(x_1 x_2 w \t \th(x_2) \i \th(x_1) \i) & \ge \ell(x_1 x_2 w \t \th(x_2) \i)-\ell(x_1) \\ &=\ell(x_1)+\ell(x_2 w \t \th(x_2) \i)-\ell(x_1) \\ &=\ell(x_2 w \t \th(x_2) \i) \ge \ell(w \t). \end{align*} Hence $w \t \in \tW_{\min}$. Thus we have $$\CP \dot \t \subset G(\pi(\t)).$$ On the other hand, let $w \in \tW_{\min}$ with central Newton point, i.e. $\<\nu_w, a\>=0$ for all roots $a$ of $G$. Then by \cite[Corollary 2.8]{HN1}, we have $w \in W_{\CP} \t$ for some standard pair $(\CP, \t)$, where $W_{\CP}$ is the Weyl group of $\CP$. Then $\CI \dot w \CI \subset \CP \dot \t$ and $G \cdot_\th \CI \dot w \CI \subset G \cdot_\th \CP \dot \t$. Thus \[G_\rig=\cup_{(\CP, \t) \in \text{StP}} \, G \cdot_\th \CP \dot \t.\qedhere\]
\end{proof}

\subsection{} For any $(\CP, \t) \in \text{StP}$, we denote by $H_R(\CP \dot \t) \subset H_R$ the $R$-submodule consisting of functions supported in $\CP \dot \t$. Note that $\CP \dot \t$ is stable under the $\th$-twisted conjugation action of $\CP$. We denote by $J_{\CP}(\CP \t)$ the set of $(\CP, \th, \omega)$-invariant distributions on $\CP \dot \t$, i.e., the set of distributions $j$ on $\CP \dot \t$ such that $j(f)=j({}^p f)$ for any $p \in \CP$ and $f \in H_R(\CP \dot \t)$. Then it is easy to see that the restriction of any $(\th, \omega)$-invariant distribution on $G$ to $\CP \dot \t$ is $(\CP, \th, \omega)$-invariant. 

\begin{theorem}\label{j-str}
The restriction map $J(G) \to \oplus_{(\CP, \t) \in \text{StP}} J_{\CP}(\CP \dot \t)$ gives a bijection from $J(G)_\rig$ to the $R$-submodule of $\oplus_{(\CP, \t) \in \text{StP}} J_{\CP}(\CP \dot \t)$ consisting of the elements $(j_{(\CP, \t)})_{(\CP, \t) \in \text{StP}} \in \oplus J_{\CP} (\CP \dot \t)$ satisfying the condition 
\[\tag{*} \forall (\CP, \t), (\CQ, \g), x \in {}^\CP \tW^\CQ, (j_{\CP, \t} \mid_{\CP \dot \t \cap \dot x \CQ \dot \g \th(\dot x) \i})={}^{\dot x}(j_{\CQ, \g} \mid_{\CQ \dot \g\cap \dot x \i \CP \dot \t \th(\dot x)}).\]
\end{theorem}

\begin{proof}
For any standard pair $(\CP, \t)$, we denote by $H_R(\CP \dot \t)$ the subset of $H_R$ consisting of functions supported in $\CP \dot \t$. By Proposition \ref{rig-g}, we have a surjection \[\tag{a} p_\rig: \oplus_{(\CP, \t) \in \text{StP}} H_R(\CP \dot \t) \twoheadrightarrow \bar H_R^\rig.\]
Therefore the restriction map $j \mapsto (j \mid_{\CP \dot \t})_{(\CP, \t) \in \text{StP}}$ from $J(G)_\rig$ to the direct sum of distributions on $\CP \dot \t$ is injective. It is also easy to see that $j \mid_{\CP \dot \t}$ is $(\CP, \th, \omega)$-invariant. Moreover, for all standard pairs $(\CP, \t), (\CQ, \g)$ and $x \in {}^{\CP} \tW^{\CQ}$, the $\th$-twisted conjugation action by $\dot x$ sends $\CQ \g \cap \dot x \i \CP \dot \t \th(\dot x)$ to $\CP \dot \t \cap \dot x \CQ \dot \g \th(\dot x) \i$. Thus $j \mid_{\CP \dot \t \cap \dot x \CQ \dot \g \th(\dot x) \i}={}^{\dot x}(j \mid_{\CQ \dot \g\cap \dot x \i \CP \dot \t \th(\dot x)})$.

On the other hand, given $(j_{(\CP, \t)})_{(\CP, \t) \in \text{StP}} \in \oplus J_{\CP} (\CP \dot \t)$ satisfying the condition $(*)$, we construct a distribution $j \in J(G)_\rig$ using the sheaf-theoretic description of distributions.  

First, we set $j \mid_{G_{\text{nrig}}}=0$. For any $g \in G_\rig$, by Proposition \ref{rig-g} we may choose a small neighborhood $U$ of $g$ such that $U \subset G_\rig$ and that there exists $h \in G$ with $h U \th(h) \i \subset \CP \dot \t$ for some standard pair $(\CP, \t)$. We define \[j \mid_U={}^{h \i} (j_{(\CP, \t)} \mid_{h U \th(h) \i}).\]

It remains to show that 

(b) The family $(j \mid_U)$ we obtained above is independent of the choice of $h$ and $(\CP, \t)$. 

Once (b) is proved, we automatically have $(j \mid_U) \mid_{U \cap U'}=(j \mid_{U'}) \mid_{U \cap U'}$ for any open subsets $U, U'$ of $G_\rig$, and thus the family $(j \mid_U)$ defines a distribution $j$ of $G$ supported in $G_\rig$. The $(\th, \omega)$-invariant condition for $j$ follows from the fact that $(j \mid_U)$ is independent of the choice of $h$. 

Now we prove (b). Let $U, U'$ be small neighborhoods of $g$ and $h, h' \in G$ with $h U \th(h) \i \subset \CP \dot \t$ and $h' U' \th(h') \i \subset \CQ \dot \g$. After $\th$-twisted conjugation, we may and do assume that $h'=1$. Then $U' \subset \CQ \dot \g$ and $\CQ \dot \g \cap h \i \CP \dot \t \th(h) \neq \emptyset$. 

We write $h$ as $h=p \dot x q$ for $p \in \CP$, $x \in {}^{\CP} \tW^{\CQ}$ and $q \in \CQ$. Then \begin{align*}
{}^{h \i} (j_{(\CP, \t)} \mid_{h U \th(h) \i}) &={}^{q \i \dot x \i}  (j_{(\CP, \t)} \mid_{\dot x q U \th(\dot x q) \i}) \quad \text{ as $j_{(\CP, \t)}$ is $(\CP, \th, \omega)$-invariant} \\ &={}^{q \i} (j_{(\CQ, \g)} \mid_{q U \th(q) \i}) \quad \text{ by the condition (*)} \\ &=j_{(\CQ, \g)} \mid_U \qquad \text{ as $j_{(\CQ, \g)}$ is $(\CQ, \th, \omega)$-invariant}.
\end{align*} 
This finishes the proof.
\end{proof}

\subsection{}\label{pqx} Let $(\CP, \t), (\CQ, \g)$ be standard pairs and $x \in {}^{\CP} \tW^{\CQ}$. We define $$H_{(\CP, \t), (\CQ, \g), x}=\{(f, -{}^{\dot x} f); f \in H_R(\CP \dot \t) \text{ with support in } \CP \dot \t \cap \dot x \CQ \dot \g \th(\dot x) \i\}.$$
We call it the $(\CP, \CQ, x)$-graph in $H_R(\CP \dot \t) \oplus H_R(\CQ \dot \g)$. 

We have seen in the proof of Theorem \ref{j-str} that $\bar H_R^\rig$ is generated by $H_R(\CP \dot \t)$. Since $\dot \t \th(\CP) \dot \t \i=\CP$, we have $[H_R(\CP), H_R(\CP \dot \t)]_{\th, \omega} \subset H_R(\CP \dot \t)$. This gives some relations in $\bar H_R^\rig$. Now we show that the remaining relations in $\bar H_R^\rig$ are given by the $(\CP, \CQ, x)$-graphs.  

\begin{theorem}\label{relation}
The kernel of the surjective map $$p_{\rig}: \oplus_{(\CP, \t) \in \text{StP}} H_R(\CP \dot \t) \twoheadrightarrow \bar H_R^\rig$$ is spanned by $[H_R(\CP), H_R(\CP \dot \t)]_{\th, \omega} \subset H_R(\CP \dot \t)$ and $H_{(\CP, \t), (\CQ, \g), x} \subset H_R(\CP \dot \t) \oplus H_R(\CQ \dot \g)$.
\end{theorem}

\begin{proof}
By definition, $[H_R(\CP), H_R(\CP \dot \t)]_{\th, \omega}$ and $H_{(\CP, \t), (\CQ, \g), x}$ are contained in the kernel of $p_{\rig}$. 

Now we prove the other direction. Let $(f_{(\CP, \t)}) \in \oplus H_R(\CP \dot \t)$ with $\sum f_{(\CP, \t)}=0$ in $\bar H_R$. Let $\Omega'=\{\t \in \Omega; f_{(\CP, \t)} \neq 0 \text{ for some } \CP\}$ and let $\Omega_0=\<\th\> \cdot \Omega'$ be the smallest $\th$-stable subset of $\Omega$ that contains $\Omega'$. Since the action of $\th$ on $\Omega$ is of finite order, $\Omega_0$ is a finite subset of $\Omega$. Let $\text{StP}_0=\{(\CP, \t) \in \text{StP}; \t \in \Omega_0\}$. As there are only finitely many standard parahoric subgroups, $\text{StP}_0$ is a finite subset of $\text{StP}$. By definition, there exists $n \in \BN$ such that $f_{(\CP, \t)} \in H_R(\CP \dot \t, \CI_n)$ for any $(\CP, \t) \in \text{StP}_0$. For any standard pair $(\CP, \t)$, $H_R(\CP \dot \t, \CI_n)$ is finite dimensional. In particular, 

(a) The restriction $\oplus_{(\CP, \t) \in \text{StP}_0} J_{\CP}(\CP \dot \t) \mid_{H_R(\CP \dot \t, \CI_n)}$ is finite dimensional. 

Let $G_\rig^0=\{g \in G_\rig; \k(g) \in \Omega_0\}$ and $G_\rig^1=G_\rig-G_\rig^0$. Let $J(G)_\rig^0$ and $J(G)_\rig^1$ be the set of $(\th, \omega)$-invariant distributions supported in $G_\rig^0$ and $G_\rig^1$, respectively. We define $(\bar H_R^\rig)_0$ and $(\bar H_R^\rig)_1$ in a similar way. Then $$J(G)_\rig=J(G)_\rig^0 \oplus J(G)_\rig^1, \bar H_R^\rig=(\bar H_R^\rig)_0 \oplus (\bar H_R^\rig)_1, \text{ and } J(G)_\rig^0=(\bar H_R^\rig)_0^*.$$

Since $\Omega_0$ is $\th$-stable, for any standard pair $(\CP, \t)$, $\CP \dot \t \subset G_\rig^0$ if $(P, \t) \in \text{StP}_0$ and $P \dot \t \subset G_\rig^1$ if $(\CP, \t) \notin \text{StP}_0$. Thus the image of $\oplus_{(\CP, \t) \in \text{StP}_0} H_R(\CP \dot \t) \to \bar H_R^\rig$ equals $(\bar H_R^\rig)_0$ and the restriction map $J(G)_\rig^0 \to J(\CP \dot \t)$ equals $0$ for any $(\CP, \t) \notin \text{StP}_0$. 

By Theorem \ref{j-str} and (a) above, the map $J(G)_\rig^0 \to \oplus_{(\CP, \t) \in \text{StP}_0} J_{\CP}(\CP \dot \t)$ is injective and there exists a finite subset $A$ of the $5$-tuples $(\CP, \t, \CQ, \g, x)$ with $(\CP, \t), (\CQ, \g) \in \text{StP}_0$ and $x \in {}^{\CP} \tW^{\CQ}$ such that 
$$\Im(J(G)_\rig \to \oplus_{(\CP, \t) \in \text{StP}_0}  J_{\CP}(\CP \dot \t) \mid_{H_R(\CP \dot \t, \CI_n)})=V \mid_{\oplus_{(\CP, \t) \in \text{StP}_0}  H_R(\CP \dot \t, \CI_n)},$$ where $V$ is the subspace of $\oplus_{(\CP, \t) \in \text{StP}_0} J_{\CP}(\CP \dot \t)$ is defined by the equations \[\tag{b} (j_{\CP, \t} \mid_{\CP \dot \t \cap \dot x \CQ \dot \g \th(\dot x) \i})={}^{\dot x}(j_{\CQ, \g} \mid_{\CQ \dot \g\cap \dot x \i \CP \dot \t \th(\dot x)}) \text{ for } (\CP, \t, \CQ, \g, x) \in A.\]

Therefore $V \mid_{\oplus_{(\CP, \t) \in \text{StP}_0}  H_R(\CP \dot \t, \CI_n)}=\Im(\oplus_{(\CP, \t) \in \text{StP}_0}  H_R(\CP \dot \t, \CI_n) \to \bar H_R^\rig)^*$. Since both spaces are finite dimensional, the kernel of the map $\oplus_{(\CP, \t) \in \text{StP}_0} H_R(\CP \dot \t, \CI_n) \to \bar H_R^\rig)$ consist precisely of the elements vanishing on $V$. 

Notice that for any standard pair $(\CP, \t)$, $$J_{\CP}(\CP \dot \t)=(H_R(\CP \dot \t)/[H_R(\CP), H_R(\CP \dot \t)]_{\th, \omega})^*.$$  So the elements of $\oplus_{(\CP, \t) \in \text{StP}_0} H_R(\CP \dot \t, \CI_n)$ vanishing on  $\oplus_{(\CP, \t) \in \text{StP}_0} J_{\CP}(\CP \dot \t)$ are exactly $\oplus_{(\CP, \t) \in \text{StP}_0} [H_R(\CP), H_R(\CP \dot \t)]_{\th, \omega} \cap H_R(\CP \dot \t, \CI_n)$.

The subspace $V$ of $\oplus_{(\CP, \t) \in \text{StP}_0} J_{\CP}(\CP \dot \t)$ is defined by the equations in (b). Taking the dual, we see that the elements vanishing on $V$ are spanned by $$\oplus_{(\CP, \t) \in \text{StP}_0}[H_R(\CP), H_R(\CP \dot \t)]_{\th, \omega} \cap H_R(\CP \dot \t, \CI_n) \text{ and } \sum_{(\CP, \t, \CQ, \g, x) \in A} H_{(\CP, \t), (\CQ, \g), x}.\qedhere$$ \end{proof}


\begin{thebibliography}{99}

\bibitem{BM}
D. Barbasch and A. Moy, \emph{A new proof of the Howe conjecture}, J. Amer. Math. Soc. {\bf 13} (2000), no. 3, 639--650. 

\bibitem{BDK}
J.~Bernstein, P.~Deligne, D.~Kazhdan,
\emph{Trace Paley-Wiener theorem for reductive $p$-adic groups},
J. d'Analyse Math. {\bf 47} (1986), 180--192.

\bibitem{BT2} 
F. Bruhat and J. Tits: {\it Groupes r\'eductifs sur un corps local:II. Sch\'emas en groupes. Existence d'une donn\'ee radicielle valu\'ee}, Inst. Hautes \'Etudes Sci. Publ. Math. {\bf 60} (1984), 197-376. 

\bibitem{CH}
D. Ciubotaru and X. He, \emph{Cocenters and representations of affine Hecke algebra}, J. Eur. Math. Soc. \textbf{19} (2017), 3143--3177.

\bibitem{CH2}
D. Ciubotaru and X. He, \emph{Cocenters of $p$-adic groups, III: Elliptic cocenter and rigid cocenter}, arXiv:1703.00378. 

\bibitem{Cl}
L. Clozel, \emph{Orbital integrals on p-adic groups: a proof of the Howe conjecture},  Ann. of Math. (2) {\bf 129} (1989), no. 2, 237--251.

\bibitem{Dat}
J.-F.~Dat,
\emph{On the $K_0$ of a $p$-adic group}, Invent. Math. {\bf 140} (2000), no. 1, 171--226. 

\bibitem{De}
S. DeBacker, \emph{Lectures on harmonic analysis for reductive $p$-adic groups}. Representations of real and $p$-adic groups, 47--94, Lect. Notes Ser. Inst. Math. Sci. Natl. Univ. Singap., 2, Singapore Univ. Press, Singapore, 2004. 

\bibitem{Ga}
R. Ganapathy, \emph{The local Langlands correspondence for $GSp_4$ over local function fields}, Amer. J. Math. {\bf 137} (2015), 1441--1534. 

\bibitem{HR}
T. Haines and M. Rapoport, \emph{On parahoric subgroups}, Adv. in Math. {\bf 219}, no. 1, (2008), 188--198; appendix to: G. Pappas, M. Rapoport, \emph{Twisted loop groups and their affine flag varieties}, Adv. in Math. {\bf 219}, no. 1, (2008), 118--198.

\bibitem{HC}
Harish-Chandra, \emph{Admissible invariant distributions on reductive p-adic groups}, Preface and
notes by S. DeBacker and P. Sally, University Lecture Series, vol. 16, American Mathematical
Society, Providence, RI, 1999.

\bibitem{HL}
G. Henniart and B. Lemaire, \emph{Repr\'esentations des espaces tordus sur un groupe r\'eductif connexe {$p$}-adique}, Ast\'erisque {\bf 386} (2017), ix+366 pp.  

\bibitem{He-14}
X. He, \emph{Geometric and homological properties of affine Deligne-Lusztig varieties},  Ann. of Math. (2) {\bf 179} (2014), 367--404.

\bibitem{He-KR}
X. He, \emph{Kottwitz-Rapoport conjecture on unions of affine Deligne-Lusztig varieties}, Ann. Sci. \`Ecole Norm. Sup. {\bf 49} (2016), 1125--1141.

\bibitem{hecke-2}
X. He, \emph{Cocenters of $p$-adic groups, II: inclusion map}, arXiv:1611.06825.

\bibitem{HN1}
X. He and S. Nie, \emph{Minimal length elements of extended affine Weyl group}, Compos. Math. {\bf 150} (2014), 1903--1927.

\bibitem{Howe}
R. Howe, \emph{Two conjectures about reductive p-adic groups}, Proc. of A.M.S. Symposia in Pure Math. XXVI (1973), 377--380.

\bibitem{howe2}
R. Howe, \emph{Harish-Chandra homomorphisms for $p$-adic groups}, volume 59 of CBMS Regional Conference Series in Mathematics. Published for the Conference Board of the Mathematical Sciences, Washington, DC, 1985.

\bibitem{Kaz}
D.~Kazhdan, \emph{Cuspidal geometry of $p$-adic groups}, J. Analyse Math. {\bf 47} (1986), 1--36.

\bibitem{Ka}
D. Kazhdan, \emph{Representations groups over close local fields}, J. Anal. Math. {\bf 47} (1986) 175--179. 

\bibitem{Ko1} R.~Kottwitz, \emph{Isocrystals with additional
structure}, Compositio Math.  \textbf{56} (1985), 201--220.

\bibitem{Ko2} R.~Kottwitz, \emph{Isocrystals with
additional structure. {II}}, Compositio Math. \textbf{109} (1997), 255--339.

\bibitem{Ku} S.~Kumar, \emph{Kac-Moody Groups, Their Flag Varieties and Representation Theory}, Progress in Mathematics, vol. 204, Birkh\"auser Boston, Inc., Boston, MA, 2002.

\bibitem{Ri}
T. Richarz, \emph{On the Iwahori-Weyl group}, Bull. Soc. Math. France \textbf{144} (2016), 117--124. 

\bibitem{Tits}
J. Tits, \emph{Reductive groups over local fields}, Automorphic forms, representations and $L$-functions (Proc. Sympos. Pure Math., Oregon State Univ., Corvallis, Ore., 1977), Part 1, pp. 29--69, Proc. Sympos. Pure Math., XXXIII, Amer. Math. Soc., Providence, R.I., 1979. 

\bibitem{Vi}
M.-F. Vign{\'e}ras, \emph{Repr\'esentations {$l$}-modulaires d'un groupe r\'eductif {$p$}-adique avec {$l\ne p$}}, Progr. Math. {\bf 137}, Birkh\"auser, Boston 1996.

\end{thebibliography}
\end{document}